\documentclass{amsart}
\usepackage{amsmath,amsthm,amssymb,amscd}

\newtheorem{theorem}{Theorem}[section]

\newtheorem{corollary}[theorem]{Corollary}
\theoremstyle{definition}
\newtheorem{definition}[theorem]{Definition}

\theoremstyle{remark}

\numberwithin{equation}{section}

\begin{document}

%
%
%
%
%
%
%
%
%

\title[Divisibility of Andrews' Singular Overpartitions by Powers of $2$ and $3$]
 {Divisibility of Andrews' Singular Overpartitions by Powers of $2$ and $3$}


\author{Rupam Barman}
\address{Department of Mathematics, Indian Institute of Technology Guwahati, Assam, India, PIN- 781039}
\email{rupam@iitg.ac.in}

\author{Chiranjit Ray}
\address{Department of Mathematics, Indian Institute of Technology Guwahati, Assam, India, PIN- 781039}
\email{chiranjitray.m@gmail.com}

\date{To appear at Research in Number Theory, Accepted on 11th June, 2019}


\subjclass{Primary 05A17, 11P83}

\keywords{Singular overpartitions; Eta-quotients; modular forms}

\dedicatory{}

\begin{abstract}
Andrews introduced the partition function $\overline{C}_{k, i}(n)$, called singular overpartition, which counts the number of overpartitions of $n$ in which no part is divisible by 
$k$ and only parts $\equiv \pm i\pmod{k}$ may be overlined. He also proved that $\overline{C}_{3, 1}(9n+3)$ and $\overline{C}_{3, 1}(9n+6)$ are divisible by $3$ for $n\geq 0$. 
Recently Aricheta proved that for an infinite family of $k$, $\overline{C}_{3k, k}(n)$ is almost always even.  
In this paper, we prove that for any positive integer $k$, $\overline{C}_{3, 1}(n)$ is almost always divisible by $2^k$ and $3^k.$
\end{abstract}

\maketitle
\section{Introduction and statement of results}
In \cite{corteel2004}, Corteel and Lovejoy introduced overpartitions. An overpartition of $n$ is a non-increasing sequence of natural numbers whose sum
is $n$ in which the first occurrence of a number may be overlined.
In \cite{andrews2015}, Andrews defined the partition function $\overline{C}_{k, i}(n)$, called singular overpartition, 
which counts the number of overpartitions of $n$ in which no part is divisible by $k$ and only parts $\equiv \pm i\pmod{k}$ may be overlined. 
For example, $\overline{C}_{3, 1}(4)=10$ with the relevant partitions being $4, \overline{4}, 2+2, \overline{2}+2, 2+1+1, \overline{2}+1+1, 2+\overline{1}+1, \overline{2}+\overline{1}+1, 1+1+1+1, \overline{1}+1+1+1$. 
For $k\geq 3$ and $1\leq i \leq \left\lfloor\frac{k}{2}\right\rfloor$, the generating function for $\overline{C}_{k, i}(n)$ is given by
\begin{align}\label{gen-fun-sop}
	\sum_{n=0}^{\infty}\overline{C}_{k, i}(n)q^n=\frac{(q^k; q^k)_{\infty}(-q^i; q^k)_{\infty}(-q^{k-i}; q^k)_{\infty}}{(q; q)_{\infty}},
\end{align}
where $\displaystyle (a; q)_{\infty}:= \prod_{j=0}^{\infty}(1-aq^j)$. Andrews proved the following Ramanujan-type congruences satisfied by $\overline{C}_{3, 1}(n)$: For $n\geq 0$,
$$\overline{C}_{3, 1}(9n+3)\equiv \overline{C}_{3, 1}(9n+6)\equiv 0 \pmod{3}.$$
Chen, Hirschhorn and Sellers \cite{chen2015} later showed that Andrews' congruences modulo $3$ are two examples of an infinite family
of congruences modulo $3$ which hold for the function $\overline{C}_{3, 1}(n)$. 
Recently, Ahmed and Baruah \cite{ahmed2015} found congruences modulo $4$, $18$ and $36$ for $\overline{C}_{3, 1}(n)$, infinite families of congruences modulo 
$2$ and $4$ for $\overline{C}_{8, 2}(n)$, congruences modulo $2$ and $3$ for $\overline{C}_{12, 2}(n)$ and $\overline{C}_{12, 4}(n)$; and congruences modulo $2$ for $\overline{C}_{24, 8}(n)$ and $\overline{C}_{48, 16}(n)$. 
\par In a recent work, Naika and Gireesh \cite{naika2016} proved congruences for $\overline{C}_{3, 1}(n)$ modulo $6$, $12$, $16$, $18$, and $24$. They also found infinite families of congruences 
for $\overline{C}_{3, 1}(n)$ modulo $12$, $18$, $48$, and $72$. In \cite{barman2018}, we affirm a conjecture of Naika and Gireesh by proving that
	$\overline{C}_{3,1}(12n+11)\equiv 0\pmod{144}$
for all $n\geq 0$.
\par 
In \cite{chen2015}, Chen, Hirschhorn and Sellers studied the parity of $\overline{C}_{k, i}(n)$. They showed that $\overline{C}_{3, 1}(n)$ is
always even and that $\overline{C}_{6, 2}(n)$ is even (or odd) if and only if $n$ is not (or is) a pentagonal number. 
In a very recent paper, Aricheta \cite{arichet2017} studied the parity of $\overline{C}_{3k, k}(n)$. To be specific, represent any positive integer $k$ as $k = 2^am$ where the integer $a \geq 0$ and $m$ is positive odd.
Assume further that $2^a\geq m$. Then Aricheta proved that $\overline{C}_{3k, k}(n)$ is almost always even, that is
\begin{align*}
    \lim_{X\to\infty} \frac{\# \left\{0<n\leq X: \overline{C}_{3k, k}(n)\equiv 0\pmod{2}\right\}}{X}&=1.
\end{align*}
\par Let $k$ be a fixed positive integer. Gordon and Ono \cite{gordon1997} proved that the number of partitions of $n$ into distinct parts is divisible by $2^k$ for almost 
all $n$. Similar studies are done for some other partition functions, for example see \cite{bringmann2008, lin2012, ray2019}.  
In this article, we study divisibility of $\overline{C}_{3, 1}(n)$ by $2^k$ and $3^k$. To be specific, we prove that $\overline{C}_{3, 1}(n)$ is divisible by $2^k$ and $3^k$ for almost all $n$.  
 \begin{theorem} \label{thm1}
 	Let $k$ be a fixed positive integer. Then $\overline{C}_{3, 1}(n)$ is almost always divisible by $2^k$, namely,
 	\begin{align*}
 	\lim_{X\to\infty} \frac{\# \left\{0<n\leq X: \overline{C}_{3, 1}(n)\equiv 0\pmod{2^k}\right\}}{X}&=1.
 	\end{align*}
 \end{theorem}
 We further prove that the partition function $\overline{C}_{3, 1}(n)$ is also divisible by $3^k$ for almost all $n$.
 \begin{theorem} \label{thm2}
 	Let $k$ be a fixed positive integer. Then $\overline{C}_{3, 1}(n)$ is almost always divisible by $3^k$, namely,
 	\begin{align*}
    \lim_{X\to\infty} \frac{\# \left\{0<n\leq X: \overline{C}_{3, 1}(n)\equiv 0\pmod{3^k}\right\}}{X}&=1.
 	\end{align*}
 \end{theorem}
Chen, Hirschhorn and Sellers \cite{chen2015} showed that $\overline{C}_{3, 1}(n)$ is even for all $n\geq 1$. Hence, we have the following corollary.
 \begin{corollary}
  Let $k$ be a fixed positive integer. Then $\overline{C}_{3, 1}(n)$ is almost always divisible by $2\cdot 3^k$.
 \end{corollary}

\section{Preliminaries}
In this section, we recall some definitions and basic facts on modular forms. For more details, see for example \cite{ono2004, koblitz1993}. We first define the matrix groups 
\begin{align*}
\text{SL}_2(\mathbb{Z}) & :=\left\{\begin{bmatrix}
a  &  b \\
c  &  d      
\end{bmatrix}: a, b, c, d \in \mathbb{Z}, ad-bc=1
\right\},\\
\Gamma_{0}(N) & :=\left\{
\begin{bmatrix}
a  &  b \\
c  &  d      
\end{bmatrix} \in \text{SL}_2(\mathbb{Z}) : c\equiv 0\pmod N \right\},
\end{align*}
\begin{align*}
\Gamma_{1}(N) & :=\left\{
\begin{bmatrix}
a  &  b \\
c  &  d      
\end{bmatrix} \in \Gamma_0(N) : a\equiv d\equiv 1\pmod N \right\},
\end{align*}
and 
\begin{align*}\Gamma(N) & :=\left\{
\begin{bmatrix}
a  &  b \\
c  &  d      
\end{bmatrix} \in \text{SL}_2(\mathbb{Z}) : a\equiv d\equiv 1\pmod N, ~\text{and}~ b\equiv c\equiv 0\pmod N\right\},
\end{align*}
where $N$ is a positive integer. A subgroup $\Gamma$ of $\text{SL}_2(\mathbb{Z})$ is called a congruence subgroup if $\Gamma(N)\subseteq \Gamma$ for some $N$. The smallest $N$ such that $\Gamma(N)\subseteq \Gamma$
is called the level of $\Gamma$. For example, $\Gamma_0(N)$ and $\Gamma_1(N)$
are congruence subgroups of level $N$. 
\par Let $\mathbb{H}:=\{z\in \mathbb{C}: \text{Im}(z)>0\}$ be the upper half of the complex plane. The group $$\text{GL}_2^{+}(\mathbb{R})=\left\{\begin{bmatrix}
a  &  b \\
c  &  d      
\end{bmatrix}: a, b, c, d\in \mathbb{R}~\text{and}~ad-bc>0\right\}$$ acts on $\mathbb{H}$ by $\begin{bmatrix}
a  &  b \\
c  &  d      
\end{bmatrix} z=\displaystyle \frac{az+b}{cz+d}$.  
We identify $\infty$ with $\displaystyle\frac{1}{0}$ and define $\begin{bmatrix}
a  &  b \\
c  &  d      
\end{bmatrix} \displaystyle\frac{r}{s}=\displaystyle \frac{ar+bs}{cr+ds}$, where $\displaystyle\frac{r}{s}\in \mathbb{Q}\cup\{\infty\}$.
This gives an action of $\text{GL}_2^{+}(\mathbb{R})$ on the extended upper half-plane $\mathbb{H}^{\ast}=\mathbb{H}\cup\mathbb{Q}\cup\{\infty\}$. 
Suppose that $\Gamma$ is a congruence subgroup of $\text{SL}_2(\mathbb{Z})$. A cusp of $\Gamma$ is an equivalence class in $\mathbb{P}^1=\mathbb{Q}\cup\{\infty\}$ under the action of $\Gamma$.
\par The group $\text{GL}_2^{+}(\mathbb{R})$ also acts on functions $f: \mathbb{H}\rightarrow \mathbb{C}$. In particular, suppose that $\gamma=\begin{bmatrix}
a  &  b \\
c  &  d      
\end{bmatrix}\in \text{GL}_2^{+}(\mathbb{R})$. If $f(z)$ is a meromorphic function on $\mathbb{H}$ and $\ell$ is an integer, then define the slash operator $|_{\ell}$ by 
$$(f|_{\ell}\gamma)(z):=(\text{det}~{\gamma})^{\ell/2}(cz+d)^{-\ell}f(\gamma z).$$
\begin{definition}
Let $\Gamma$ be a congruence subgroup of level $N$. A holomorphic function $f: \mathbb{H}\rightarrow \mathbb{C}$ is called a modular form with integer weight $\ell$ on $\Gamma$ if the following hold:
\begin{enumerate}
 \item We have $$f\left(\displaystyle \frac{az+b}{cz+d}\right)=(cz+d)^{\ell}f(z)$$ for all $z\in \mathbb{H}$ and all $\begin{bmatrix}
a  &  b \\
c  &  d      
\end{bmatrix} \in \Gamma$.
\item If $\gamma\in \text{SL}_2(\mathbb{Z})$, then $(f|_{\ell}\gamma)(z)$ has a Fourier expansion of the form $$(f|_{\ell}\gamma)(z)=\displaystyle\sum_{n\geq 0}a_{\gamma}(n)q_N^n,$$
where $q_N:=e^{2\pi iz/N}$.
\end{enumerate}
\end{definition}
For a positive integer $\ell$, the complex vector space of modular forms of weight $\ell$ with respect to a congruence subgroup $\Gamma$ is denoted by $M_{\ell}(\Gamma)$.
\begin{definition}\cite[Definition 1.15]{ono2004}
	If $\chi$ is a Dirichlet character modulo $N$, then we say that a modular form $f\in M_{\ell}(\Gamma_1(N))$ has Nebentypus character $\chi$ if
	$$f\left( \frac{az+b}{cz+d}\right)=\chi(d)(cz+d)^{\ell}f(z)$$ for all $z\in \mathbb{H}$ and all $\begin{bmatrix}
	a  &  b \\
	c  &  d      
	\end{bmatrix} \in \Gamma_0(N)$. The space of such modular forms is denoted by $M_{\ell}(\Gamma_0(N), \chi)$. 
\end{definition}
Recall that Dedekind's eta-function $\eta(z)$ is defined by
\begin{align*}
	\eta(z):=q^{1/24}(q;q)_{\infty}=q^{1/24}\prod_{n=1}^{\infty}(1-q^n),
\end{align*}
where $q:=e^{2\pi iz}$ and $z\in \mathbb{H}$. A function $f(z)$ is called an eta-quotient if it is of the form
\begin{align*}
	f(z)=\prod_{\delta\mid N}\eta(\delta z)^{r_\delta},
\end{align*}
where $N$ is a positive integer and $r_{\delta}$ is an integer. 
\par 
We now recall two theorems from \cite[p. 18]{ono2004} which will be used to prove our results.
\begin{theorem}\cite[Theorem 1.64]{ono2004}\label{thm_ono1} If $f(z)=\prod_{\delta\mid N}\eta(\delta z)^{r_\delta}$ 
	is an eta-quotient such that $\ell=\frac{1}{2}\sum_{\delta\mid N}r_{\delta}\in \mathbb{Z}$, 
	$$\sum_{\delta\mid N} \delta r_{\delta}\equiv 0 \pmod{24}$$ and
	$$\sum_{\delta\mid N} \frac{N}{\delta}r_{\delta}\equiv 0 \pmod{24},$$
	then $f(z)$ satisfies $$f\left( \frac{az+b}{cz+d}\right)=\chi(d)(cz+d)^{\ell}f(z)$$
	for every  $\begin{bmatrix}
	a  &  b \\
	c  &  d      
	\end{bmatrix} \in \Gamma_0(N)$. Here the character $\chi$ is defined by $\chi(d):=\left(\frac{(-1)^{\ell} s}{d}\right)$, where $s:= \prod_{\delta\mid N}\delta^{r_{\delta}}$. 
\end{theorem}
Suppose that $f$ is an eta-quotient satisfying the conditions of Theorem \ref{thm_ono1} and that the associated weight $\ell$ is a positive integer. 
If $f(z)$ is holomorphic at all of the cusps of $\Gamma_0(N)$, then $f(z)\in M_{\ell}(\Gamma_0(N), \chi)$. The following theorem gives the necessary criterion for determining orders of an eta-quotient at cusps.
\begin{theorem}\cite[Theorem 1.65]{ono2004}\label{thm_ono2}
	Let $c, d$ and $N$ be positive integers with $d\mid N$ and $\gcd(c, d)=1$. If $f$ is an eta-quotient satisfying the conditions of Theorem \ref{thm_ono1} for $N$, then the 
	order of vanishing of $f(z)$ at the cusp $\frac{c}{d}$ 
	is $$\frac{N}{24}\sum_{\delta\mid N}\frac{\gcd(d,\delta)^2r_{\delta}}{\gcd(d,\frac{N}{d})d\delta}.$$
\end{theorem}
 
\section{Proof of Theorems \ref{thm1} and \ref{thm2}}
The generating function of $\overline{C}_{3, 1}(n)$ is given by 
\begin{align}\label{4}
\sum_{n=0}^{\infty}\overline{C}_{3, 1}(n)q^n=\frac{(q^2; q^2)_{\infty}(q^3; q^3)_{\infty}^2}{(q; q)_{\infty}^2(q^6; q^6)_{\infty}}.
\end{align}
We note that  $\eta(24z)= q\prod_{n=1}^{\infty}(1-q^{24n})$ is a power series of $q$. Given a prime $p$, let 
\begin{align*}
A_p(z) = \prod_{n=1}^{\infty} \frac{(1-q^{24n})^p}{(1-q^{24pn})} = \frac{\eta^p(24z)}{\eta(24pz)}. 
\end{align*}
Then using binomial theorem we have 
\begin{align}\label{4.1}
A_p^{p^k}(z) = \frac{\eta^{p^{k+1}}(24z)}{\eta^{p^k}(24pz)} \equiv 1 \pmod {p^{k+1}}.
\end{align}
Define $B_{p,k}(z)$ by
\begin{align}\label{4.2}
B_{p,k}(z) = \left(\frac{\eta(48z)\eta(72z)^2}{\eta(24z)^2\eta(144z)}\right)A_p^{p^k}(z).
\end{align}
Modulo $p^{k+1}$, we have
\begin{align}\label{new-110}
B_{p,k}(z) \equiv \frac{\eta(48z)\eta(72z)^2}{\eta(24z)^2\eta(144z)} = \frac{(q^{48}; q^{48})_{\infty}(q^{72}; q^{72})_{\infty}^2}{(q^{24}; q^{24})_{\infty}^2(q^{144}; q^{144})_{\infty}}.
\end{align}
Combining \eqref{4} and \eqref{new-110}, we obtain  
\begin{align}\label{4.3}
B_{p,k}(z) \equiv \sum_{n=0}^{\infty}\overline{C}_{3, 1}(n)q^{24n} \pmod {p^{k+1}}.
\end{align}
\par We now give a sketch of the idea of the proof of our main results. Firstly, we show that the eta-quotient $B_{p,k}(z)$ is a modular form on $\Gamma_0(576)$ with some Nebentypus character for $p=2, 3$.
Then we apply a deep theorem of Serre regarding the divisibility of Fourier coefficients of modular forms to $B_{2,k}(z)$ and $B_{3,k}(z)$. Finally, we use the congruence $\eqref{4.3}$ to deduce 
divisibility properties of $\overline{C}_{3, 1}(n)$.
\begin{proof}[Proof of Theorem \ref{thm1}]
We put $p=2$ in \eqref{4.2} to obtain  
\begin{align}\label{4.2.1}
B_{2,k}(z) = \left(\frac{\eta(48z)\eta(72z)^2}{\eta(24z)^2\eta(144z)}\right)A_2^{2^k}(z)=\frac{\eta(24z)^{2^{k+1}-2}~\eta(72z)^2}{\eta(48z)^{2^k-1}~\eta(144z)}.
\end{align}
Now, $B_{2, k}$ is an eta-quotient with $N=576$. 
The cusps of $\Gamma_{0}(576)$ are represented by fractions $\frac{c}{d}$ where $d\mid 576$ and 
$\gcd(c, d)=1$. For example, see \cite[p. 5]{ono1996}. By Theorem \ref{thm_ono2}, we find that $B_{2,k}(z)$ is holomorphic at a cusp $\frac{c}{d}$ if and only if
\begin{align*}
\frac{\gcd(d,24)^2}{24} \left(2^{k+1}-2\right) + \frac{\gcd(d,48)^2}{48}\left(1-2^{k}\right)+\frac{\gcd(d,72)^2}{36}-\frac{\gcd(d,144)^2}{144}\geq 0.
\end{align*}
Equivalently, if and only if
\begin{align*}
D:= 6\frac{\gcd(d,24)^2}{\gcd(d,144)^2} \left(2^{k+1}-2\right) + 3\frac{\gcd(d,48)^2}{\gcd(d,144)^2}\left(1-2^{k}\right)+4\frac{\gcd(d,72)^2}{\gcd(d,144)^2}-1\geq 0.
\end{align*}
In the following table, we find all the possible values of $D$.
\begin{center}
\begin{tabular}{|p{3.1cm}|p{1.9cm}|p{1.9cm}|p{1.9cm}|p{1.2cm}|}
	\hline
	~~$d\mid 576$ & $\dfrac{\gcd(d,24)^2}{\gcd(d,144)^2}$ &$\dfrac{\gcd(d,48)^2}{\gcd(d,144)^2}$&$\dfrac{\gcd(d,72)^2}{\gcd(d,144)^2}$& $D$\\	
	\hline	
    ~~$1, 2, 3, 4, 6, 8, 12, 24$   &~~ $1$   & ~~$1$ & ~~  $1$ &~~$9\cdot 2^k-6$\\
	\hline
	~~$16, 32, 48, 64, 96, 192$&~~ $1/4$ &~~ $1$   &~~ $ 1/4$ &~~$0$\\
	\hline
	~~$9, 18, 36, 72$ &~~ $1/9$ & ~~ $1/9$ & ~~ $1$ &~~$2^k+2$\\
	\hline
	~~$144, 288, 576$    &~~ $1/{36}$ & ~~ $1/9$ & ~~  $1/4$ &~~0\\
	\hline
\end{tabular}
\end{center}	
\vspace{.5cm}
Since $D\geq 0$ for all $d\mid 576$, we have that $B_{2,k}(z)$ is holomorphic at every cusp $\frac{c}{d}$.
Using Theorem \ref{thm_ono1}, we find that the weight of $B_{2,k}(z)$ is $\ell=2^{k-1}$. Also, the associated character for $B_{2,k}(z)$ is given by $\chi_1=(\frac{(-1)^{\ell}2^{2^{k+1}}3^{2^k+1}}{\bullet})$.
Finally, Theorem \ref{thm_ono1} yields that $B_{2,k}(z) \in M_{2^{k-1}}\left(\Gamma_{0}(576), \chi_1\right)$.
Let $m$ be a positive integer. By a deep theorem of Serre \cite[p. 43]{ono2004}, if $f(z)\in M_{\ell}(\Gamma_0(N), \chi)$ has Fourier expansion 
$$f(z)=\sum_{n=0}^{\infty}c(n)q^n\in \mathbb{Z}[[q]],$$
then there is a constant $\alpha>0$  such that
$$ \# \left\{n\leq X: c(n)\not\equiv 0 \pmod{m} \right\}= \mathcal{O}\left(\frac{X}{(\log{}X)^{\alpha}}\right).$$
Since $B_{2,k}(z) \in M_{2^{k-1}}\left(\Gamma_{0}(576), \chi_1\right)$, the Fourier coefficients of $B_{2,k}(z)$ are almost always divisible by $m=2^k$. Now, using \eqref{4.3}   
we complete the proof of the theorem.
\end{proof}
\begin{proof}[Proof of Theorem \ref{thm2}] We put $p=3$ in \eqref{4.2} to obtain  
\begin{align}\label{4.3.1}
B_{3,k}(z) = \left(\frac{\eta(48z)\eta(72z)^2}{\eta(24z)^2\eta(144z)}\right)A_3^{3^k}(z)=\frac{\eta(24z)^{3^{k+1}-2}~\eta(48z)}{\eta(72z)^{3^k-2}~\eta(144z)}.
\end{align}
As before, the cusps of $\Gamma_{0}(576)$ are represented by fractions $\frac{c}{d}$ where $d\mid 576$ and 
$\gcd(c, d)=1$. By Theorem \ref{thm_ono2}, $B_{3,k}(z)$ is holomorphic at a cusp $\frac{c}{d}$ if and only if
\begin{align*}
\frac{\gcd(d,24)^2}{24} \left(3^{k+1}-2\right) + \frac{\gcd(d,48)^2}{48}+\frac{\gcd(d,72)^2}{72}\left(2-3^{k}\right)-\frac{\gcd(d,144)^2}{144}\geq 0.
\end{align*}
Equivalently, if and only if
\begin{align*}
L:= 6\frac{\gcd(d,24)^2}{\gcd(d,144)^2} \left(3^{k+1}-2\right) + 3\frac{\gcd(d,48)^2}{\gcd(d,144)^2}+2\frac{\gcd(d,72)^2}{\gcd(d,144)^2}\left(2-3^{k}\right)-1\geq 0.
\end{align*}
Using the table used to evaluate the values of $D$, we find that $L\geq 0$ for all $d\mid 576$. As before, using Theorem \ref{thm_ono1} we find that $B_{3,k}(z) \in M_{3^{k}}\left(\Gamma_{0}(576), \chi_2\right)$, 
where the character $\chi_2$ is given by $\chi_2=(\frac{-2^{2\cdot 3^{k+1}}3^{3^k+1}}{\bullet})$. 
Using the same reasoning and \eqref{4.3}, we find that $\overline{C}_{3, 1}(n)$ is divisible by $3^k$ for almost all $n\geq 0$.
This completes the proof of the theorem.
\end{proof}

\end{document}